\documentclass[reqno]{amsart}
\usepackage{amsmath, amssymb, amsthm, epsfig}
\usepackage{hyperref, latexsym}
\usepackage{url}
\usepackage[mathscr]{euscript}
\usepackage{mathabx}
\usepackage{color}
\usepackage{fullpage} 
\usepackage{setspace}

\allowdisplaybreaks

\def\today{\ifcase\month\or
  January\or February\or March\or April\or May\or June\or
  July\or August\or September\or October\or November\or December\fi
  \space\number\day, \number\year}

 \newtheorem{theorem}{Theorem}
  
 \newtheorem{lemma}[theorem]{Lemma}

 \theoremstyle{definition}

 \theoremstyle{remark}

 \newcommand{\C}{\mathbb{C}}
 \newcommand{\R}{\mathbb{R}}

\newcommand{\tf}{\widehat{f}}

 \newcommand{\du}{\text{\rm d}u}
 \newcommand{\dv}{\text{\rm d}v}

 \newcommand{\dx}{\text{\rm d}x}
 
 \newcommand{\dy}{\text{\rm d}y}

\newcommand{\im}{{\rm Im}\,}
\newcommand{\re}{{\rm Re}\,}

\def\={\;=\;}  \def\:{\;:=\;}  \def\+{\,+\,}  \def\m{\,-\,}    \def\t{\hskip .5  pt}  \def\tt{\hskip 1  pt}  
 \def\tf{\tfrac}

\makeatletter
\@namedef{subjclassname@2020}{\textup{2020} Mathematics Subject Classification}
\makeatother

\begin{document}
\title{On Littlewood's estimate for the modulus of the \\ zeta function  on the critical line}

\author[Carneiro and Milinovich]{Emanuel Carneiro and Micah B. Milinovich}
\subjclass[2020]{11M06, 11M26, 41A30}
\keywords{Riemann zeta-function, bandlimited approximations}

\address{
The Abdus Salam International Centre for Theoretical Physics,
Strada Costiera, 11, I - 34151, Trieste, Italy}

\email{carneiro@ictp.it}

\address{Department of Mathematics, University of Mississippi, University, MS 38677 USA}

\email{mbmilino@olemiss.edu}

\allowdisplaybreaks
\numberwithin{equation}{section}

\maketitle  

\begin{abstract}
Inspired by a result of Soundararajan, assuming the Riemann hypothesis (RH), we prove a new inequality for the logarithm of the modulus of the Riemann zeta-function on the critical line in terms of a Dirichlet polynomial over primes and prime powers. Our proof uses the Guinand-Weil explicit formula in conjunction with extremal one-sided bandlimited approximations for the Poisson kernel. As an application, by carefully estimating the Dirichlet polynomial, we revisit a 100-year-old estimate of Littlewood and give a slight refinement of the sharpest known upper bound (due to Chandee and Soundararajan) for the modulus of the zeta function on the critical line assuming RH, by providing explicit lower-order terms.
\end{abstract}

%\tableofcontents

\section{Introduction}
Let $\zeta(s)$ denote the Riemann zeta-function. Assuming the Riemann hypothesis (RH), a classical estimate of Littlewood \cite{Littlewood} from 1924 states that, for sufficiently large $t$, there is a constant $C>0$ such that
\[
\big|\zeta(\tfrac{1}{2}+it)\big| \ll \exp\Big( C \frac{\log t}{\log\log t} \Big).
\]
In the past 100 years, no improvement on this estimate has been made apart from reducing the permissible values of $C$. With Littlewood's estimate in mind, in this paper we prove a variation of an inequality for $\log|\zeta(\frac{1}{2}+it)|$ due to Soundararajan in terms of a Dirichlet polynomial over the primes and prime powers. Assuming RH, for sufficiently large $t$, Soundararajan \cite[Main Proposition]{Sound} proved that
\begin{equation} \label{Sound_ineq}
\log\big|\zeta(\tfrac{1}{2}+it)\big| \le \mathrm{Re} \sum_{n \le x} \frac{\Lambda(n)}{n^{1/2+\lambda/\log x +it}\log n} \frac{\log(x/n)}{\log x} + \frac{(1\!+\!\lambda)}{2} \frac{\log t}{\log x} + O\left( \frac{1}{\log x} \right),
\end{equation}
for $2 \le x \le t^2$ and $\lambda \ge \lambda_0 = 0.4912 \ldots$, where $\lambda_0$ denotes the unique positive real number satisfying $e^{-\lambda_0} = \lambda_0+\lambda_0^2/2$. Here, as usual, we let $\Lambda(n)$ denote the von Mangoldt function defined to be $\log p$, if $n=p^m$ with $p$ a prime number and $m\geq 1$ an integer, and to be zero otherwise. 

\smallskip

The usefulness of Soundararajan's inequality \eqref{Sound_ineq} is that there is considerable flexibility in choosing the parameter $x$. For example, assuming RH, choosing $x=(\log t)^{2-\varepsilon}$ for any $\varepsilon>0$, and estimating the sum over $n$ trivially, Soundararajan \cite[Corollary C]{Sound} deduced that
\[
\big|\zeta(\tfrac{1}{2}+it)\big| \ll \exp\!\left( \bigg( \frac{1\!+\!\lambda_0}{4} +o(1) \bigg) \frac{\log t}{\log\log t} \right) \le \exp\!\left( \frac{3}{8} \frac{\log t}{\log\log t} \right)
\]
for sufficiently large $t$. At the time, this was the sharpest known version of Littlewood's result, improving on earlier work of Ramachandra and Sankaranarayanan \cite{RamSank}. The flexibility in the parameter $x$ in \eqref{Sound_ineq} also allowed Soundararajan \cite[Main Theorem]{Sound} to study the frequency of large values of $|\zeta(\frac{1}{2}+it)|$ and allowed Harper \cite[Theorem 1]{Harper} to give sharp upper bounds for $2k$-th moment of the zeta function on the critical line; see also \cite[Corollary A]{Sound}. An overview of these (and other related) ideas concerning the distribution of values of zeta and $L$-functions can be found in the recent survey article of Soundararajan \cite{Sound2}. 

\smallskip

In \cite{Sound}, Soundararajan asked for an upper bound for $|\zeta(\frac{1}{2}+it)|$ on RH that attained the limit of existing methods. Using the Guinand-Weil explicit formula, it was shown in \cite{CS} that this problem could be framed in terms of bandlimited minorants of the function $\log((4+x^2)/x^2)$. Assuming RH, drawing upon the work of Carneiro and Vaaler \cite{CV}, Chandee and Soundararajan \cite[Theorem 1.1]{CS} use the optimal such minorants to prove that 
\begin{equation} \label{20231227_20:31}
\log\big|\zeta(\tfrac{1}{2}+it)\big| \leq \left( \frac{\log 2}{2} + o(1)\right) \frac{\log t }{\log \log t}\,,
\end{equation}
as $t \to \infty$. They initially proved that the term of $o(1)$ is $O(\log \log \log t/ \log \log t)$ but it was later observed in \cite{CC} and \cite{CCM} that the term of $o(1)$ can be taken to be $O(1/\log \log t)$. 

\smallskip

The goal of this note is to use the Guinand-Weil explicit formula, in conjunction with extremal bandlimited majorants and minorants for the Poisson kernel constructed in \cite{CCM, CLV}, to give an analogue of Soundararajan's inequality  for $\log|\zeta(\frac{1}{2}+it)|$ in \eqref{Sound_ineq} and then to use this new inequality to give a slight refinement of Chandee and Soundararajan's bound in \eqref{20231227_20:31}. For example, assuming RH, we show that
\[
\log\big|\zeta(\tfrac{1}{2}+it)\big| \leq \frac{\log 2}{2} \frac{\log t }{\log \log t} + \bigg(\frac{\log 2}{2} + \log^2 2\bigg) \frac{\log t }{(\log \log t)^2} + O\!\left( \frac{\log t }{(\log \log t)^3} \right),
\]
as $t \to \infty$. Moreover, our new analogue of \eqref{Sound_ineq} has an explicit weight function in the Dirichlet polynomial over primes and prime powers, and it maintains the flexibility in the parameter $x$. So it potentially remains useful in applications such as studying the frequency of large values of $|\zeta(\frac{1}{2}+it)|$. 

\smallskip

In order to state our results, with the range $0 \leq u <1$ in mind, we define the following special function:
\begin{equation}\label{20231127_13:35}
F(u) :=  \int_0^\infty  \frac{\sinh( 2 u y)}{\cosh^2 (y)} \dy = \frac{\pi u}{\sin(\pi u)} - u \left( \frac{\Gamma'}{\Gamma}\Big( \frac{u+1}{2} \Big)-\frac{\Gamma'}{\Gamma}\Big( \frac{u}{2} \Big) \right) +1.
\end{equation}
The identity in \eqref{20231127_13:35} follows from \cite[Eq.~3.512.1 and 3.541.8]{GR}. Note that $F(0)=0$ and that $F$ is in fact analytic in the open ball of radius $1$ centered at $u=0$, with simple poles at $u = \pm1$. Using standard facts about the gamma function and the cosecant function (e.g., \cite[Eq.~1.411.11 and 8.374]{GR}), one has the series expansion
\begin{equation}\label{series}
F(u) = 2 \log 2 \cdot u + 2 \sum_{k=1}^\infty \Big(1\!-\!\frac{1}{2^{2k}}\Big) \, \zeta(2k\!+\!1) \, u^{2k+1}, \quad \text{for } |u|<1.
\end{equation}
Our main result is the following inequality.

\begin{theorem}\label{new_Sound}  Assume RH. For $t\ge 10$ and $x\ge 2$, we have
\begin{equation}\label{main_ineq}
\begin{split}
\log\big|\zeta(\tfrac{1}{2}+it)\big| &\le  \mathrm{Re} \sum_{n\leq x}\dfrac{\Lambda(n)}{n^{1/2+it} \log x} \, F\!\left( \frac{\log(x/n)}{\log x} \right) + \log 2 \cdot \frac{\log t}{\log x}    + O\!\left( \frac{\sqrt{x}\log x}{t} 
+ 1 \right).
\end{split}
\end{equation}
\end{theorem}

\noindent {\sc Remarks.} 

(i) For $2 \le n \le x$, we have the bounds
\begin{equation}\label{weight_ineq}
0 \le \frac{1}{\log x} \cdot F\!\left( \frac{\log(x/n)}{\log x} \right) \le \frac{1}{\log n} - \frac{1}{\log (x^2/n)} \leq\min\left\{ \frac{1}{\log n}, \frac{2 \, \log(x/n) }{\log^2 n}\right\}.
\end{equation}
To see this, note that by using the elementary inequality $\cosh(x) \ge e^{x}/2$ for $x\ge0$ in the integral formulation of $F$, one arrives at the bound
\begin{equation}\label{20231129_16:28}
F(u) \le 2 \int_0^\infty \left(e^{-2(1-u)x} + e^{-2(1+u)x} \right) \dx =\frac{1}{(1-u)} - \frac{1}{(1+u)} = \frac{2u}{1 - u^2}.
\end{equation}
Then, \eqref{weight_ineq} plainly follows from \eqref{20231129_16:28}.

\smallskip

(ii) A similar inequality appears in the work of Chandee and Soundararajan \cite[Equation (5)]{CS}. Letting $x=e^{2 \pi \Delta}$ in their work, for $t\ge 10$ and $x\ge 2$, they show that
\[
\log\big|\zeta(\tfrac{1}{2}+it)\big| \le  \mathrm{Re} \sum_{n\leq x}\dfrac{\Lambda(n)}{n^{1/2+it}} \, W(n;x) + \frac{\log t}{\log x} \log\Big( \frac{2}{1+x^{-2}} \Big)    + O\!\left( \frac{\sqrt{x}\log x}{t} + 1 \right),
\]
where the weight function $W(n;x)$ is not given explicitly but is shown to satisfy the bound $|W(n;x)| \ll 1$ for $2\le n\le x$. This suffices for the pointwise bound for $|\zeta(\tfrac{1}{2}+it)\big|$ in \eqref{20231227_20:31} but is not amenable to applications such as studying the distribution of large values of $\zeta(s)$ on the critical line.

\smallskip

Using Theorem \ref{new_Sound}, we obtain a slight refinement of \eqref{20231227_20:31} with lower-order terms.

\begin{theorem}\label{lower-order}
Assume RH. For $t\ge 10$ and any integer $K\ge 4$, we have
\[
\begin{split}
\log\big|\zeta(\tfrac{1}{2}+it)\big| & \, \le \, \frac{\log 2}{2} \frac{\log t }{\log \log t} +  \bigg(\frac{\log 2}{2} + \log^2 2\bigg) \frac{\log t }{(\log \log t)^2} + \Big( 2 \log^2 2 + 2 \log^3 2 \Big) \frac{\log t }{(\log \log t)^3}
\\
&\quad \qquad + \sum_{k=4}^K C_k \cdot \frac{\log t}{(\log \log t)^{k}} + O_K\bigg( \frac{\log t }{(\log \log t)^{K+1}}  \bigg),
\end{split}
\]
where the constants $C_k$ are effectively computable as described in \S \ref{Optimal_x}.

\end{theorem}

\noindent {\sc Remark.} Within our setup, there are different ways to arrive at an inequality of the form 
\begin{equation*}
\log\big|\zeta(\tfrac{1}{2}+it)\big|  \leq \sum_{k=1}^K C_k \cdot \frac{\log t}{(\log \log t)^{k}} + O_K\bigg( \frac{\log t }{(\log \log t)^{K+1}}  \bigg)\,,
\end{equation*}
with $C_1 = (\log 2)/2$ and each $C_k$ explicit for $2 \le k \leq K$; see, for instance, equation \eqref{20240102_13:27}. We go one step further and address the problem of doing this in an optimal way (within our framework). Running our process to obtain the values of $C_k$, in addition to the values of $C_1, C_2$, and $C_3$ stated in Theorem \ref{lower-order}, we arrive at (set $L:= \log 2$)
\[
C_4 =-L+ 6 L^3 + 4 L^4 + \dfrac{9 \, \zeta(3)}{4}\, ; \quad C_5 = -\frac{4 L}{3} - 8 L^2 + 16 L^4 + 8 L^5+ \dfrac{9 \, \zeta(3)}{2} + 18 \, \zeta(3) L\, ; 
\]
and
\[
C_6 = \frac{4L}{3} - \frac{40 L^2}{3}  - 40 L^3 +  40 L^5 + 16L^6 
 +  \bigg(- 9 + 45 L + 90 L^2 - \frac{81 \, \zeta(3)}{16 L} \bigg) \zeta(3) + \dfrac{225 \, \zeta(5)}{4}.
\]

\section{Preliminary lemmas}
We denote a generic non-trivial zero of $\zeta(s)$ by $\rho$ and, assuming RH, we write $\rho = \frac12 + i \gamma$. Unconditionally, for $s\ne 1$ and $s$ not coinciding with a zero of $\zeta(s)$, the partial fraction decomposition for $\zeta'(s)/\zeta(s)$ (cf.~\cite[Chapter 12]{Dav}) states that
\begin{align*}
	\dfrac{\zeta'}{\zeta}(s) & = \displaystyle\sum_{\rho}\bigg(\dfrac{1}{s-\rho}+\frac{1}{\rho}\bigg)-\dfrac{1}{2}\dfrac{\Gamma'}{\Gamma}\bigg(\dfrac{s}{2}+1\bigg)+B+\frac{1}{2}\log\pi - \frac{1}{s-1},
	\end{align*}
where $B=-\sum_{\rho}\re(1/\rho)$. Assuming RH, for $\beta > 0$ and $t \ge 1$, using Stirling's formula for the gamma function, one obtains
\begin{align}\label{partial_fraction_dec}
	\mathrm{Re} \, \frac{\zeta'}{\zeta}(\tfrac{1}{2}+\beta+it) = - \dfrac{1}{2}\log\frac{t}{2\pi} + \sum_{\gamma} h_\beta(t-\gamma)+O\bigg(\frac{1}{t}\bigg), 
\end{align}
where 
\begin{equation}\label{Def_Poisson_kernel_beta}
h_{\beta}(x) := \frac{\beta}{\beta^2 + x^2}
\end{equation}
is the Poisson kernel.

\subsection{Extremal bandlimited approximations} Recall that an entire function $f: \C \to \C$ is said to be of {\it exponential type} if 
$$\tau(f) := \limsup_{|z| \to \infty} \,|z|^{-1}  \log |f(z)|  < \infty.$$
In this case, the number $\tau(f)$ is called the {\it exponential type} of $f$. An entire function $f:\C \to \C$ is said to be {\it real entire} if its restriction to $\R$ is real-valued. If $f \in L^1(\R)$ our normalization for the Fourier transform is as follows:
$$\widehat{f}(\xi) = \int_{-\infty}^{\infty} f(x) \,e^{-2 \pi i x \xi}\,\dx.$$
This extends to a unitary operator in $L^2(\R)$ and functions that have compactly supported Fourier transform are called {\it bandlimited} functions. In this context, the classical Paley-Wiener theorem is a result that serves as a bridge between complex analysis and Fourier analysis. It says that, for $f \in L^2(\R)$, the following two conditions are equivalent: (i) ${\rm supp}(\widehat{f}) \subset [-\Delta, \Delta]$ and; (ii) $f$ is equal a.e.~on $\R$ to the restriction of an entire function of exponential type at most $2 \pi \Delta$.

\smallskip

The so-called Beurling-Selberg extremal problem in approximation theory is concerned with finding one-sided bandlimited approximations to a given function $f:\R \to \R$, in a way that $L^1(\R)$-error is minimized. Our first lemma is a reproduction of \cite[Lemma 9]{CCM} due to Carneiro, Chirre, and Milinovich. It presents the extremal Beurling-Selberg majorants and minorants for the Poisson kernel. This construction is derived from the general Gaussian subordination framework of Carneiro, Littmann, and Vaaler \cite{CLV}.

\begin{lemma}[Extremal functions for the Poisson kernel]\label{lemma_extr_Poisson} Let $\beta >0$ be a real number and let $\Delta >0$ be a real parameter. Let $h_{\beta}:\R \to \R$ be defined as in \eqref{Def_Poisson_kernel_beta}. Then there is a unique pair of real entire functions $m_{\beta,\Delta}^{\pm}:\mathbb{C}\to\mathbb{C}$ satisfying the following properties:
	\begin{itemize}
	\item[(i)] The real entire functions $m_{\beta,\Delta}^{\pm}$ have exponential type at most $2\pi\Delta$.
	
	\smallskip
	
	\item[(ii)] The inequalities
	$$m_{\beta,\Delta}^{-}(x) \leq h_{\beta}(x) \leq m_{\beta,\Delta}^{+}(x)$$
hold pointwise for all $x \in \R$.

\smallskip

	\item[(iii)] Subject to conditions {\rm(i)} and {\rm (ii)}, the value of the integral 
$$ \int_{-\infty}^{\infty}\big\{m^{+}_{\beta,\Delta}(x)-m^{-}_{\beta,\Delta}(x)\big\}\,\dx$$
is minimized.
\end{itemize}

\smallskip

\noindent The functions $m_{\beta,\Delta}^{\pm}$ are even and verify the following additional properties:

\smallskip

	\begin{itemize}
	 \item[(iv)]The $L^1-$distances of $m_{\beta,\Delta}^{\pm}$ to $h_{\beta}$ are explicitly given by
	 \begin{equation*}
		\int_{-\infty}^{\infty}\big\{m^{+}_{\beta,\Delta}(x)-h_{\beta}(x)\big\}\,\dx=\dfrac{2\pi e^{-2\pi\beta\Delta}}{1-e^{-2\pi\beta\Delta}}
		\end{equation*}
		and
		\begin{equation*}
		\int_{-\infty}^{\infty}\big\{h_{\beta}(x)-m^{-}_{\beta,\Delta}(x)\big\}\,\dx=\dfrac{2\pi e^{-2\pi\beta\Delta}}{1+e^{-2\pi\beta\Delta}}.
		\end{equation*}
		
		\item[(v)] The Fourier transforms of $m_{\beta,\Delta}^{\pm}$ are even continuous functions supported on the interval $[-\Delta,\Delta]$ and given by $($for $|\xi| \leq \Delta$$)$
		\begin{equation} \label{FT_maj_min_Poisson}
		\widehat{m}_{\beta,\Delta}^{\pm}(\xi) = \pi \left(\dfrac{e^{2\pi\beta(\Delta - |\xi|)}-e^{-2\pi\beta(\Delta-|\xi|)}}{\left(e^{\pi\beta\Delta}\mp e^{-\pi\beta\Delta}\right)^2}\right).
		\end{equation}
		
		\smallskip
		
		\item[(vi)] The functions $m_{\beta,\Delta}^{\pm}$ are explicitly given by
		\begin{equation*}
		m_{\beta,\Delta}^{\pm}(z) = \left(\frac{\beta}{\beta^2 + z^2}\right) \left( \frac{e^{2\pi \beta \Delta} +  e^{-2\pi \beta \Delta} - 2\cos(2\pi\Delta z)}{\left(e^{\pi\beta\Delta}\mp e^{-\pi\beta\Delta}\right)^2} \right).
		\end{equation*}
In particular, the function $m_{\beta,\Delta}^{-}$ is non-negative on $\R$.	

\smallskip

	\item[(vii)] Assume that $0 < \beta \leq 1$ and $\Delta \geq 1$. For any real number $x$ we have
	\begin{equation}\label{decay_square_maj_Poisson}
	 0 < m_{\beta,\Delta}^{-}(x) \leq h_{\beta}(x) \leq  m_{\beta,\Delta}^{+}(x) \ll \frac{1}{\beta(1 + x^2)},
	 \end{equation}
and, for any complex number $z=x+iy$, we have
		\begin{align}\label{Complex_est_maj}
		\big|m_{\beta,\Delta}^{+}(z)\big|\ll \frac{\Delta^{2}e^{2\pi\Delta|y|}}{\beta(1+\Delta|z|)} 
		\end{align}
		and
		\begin{align}\label{Complex_est_min}
		\big|m_{\beta,\Delta}^{-}(z)\big|\ll \dfrac{\beta \Delta^{2}e^{2\pi\Delta|y|}}{1+\Delta|z|},   
		\end{align}
where the constants implied by the $\ll$ notation are universal.	

\end{itemize}
\end{lemma}

\noindent {\sc Remark.} Part (vii) of the the previous lemma is stated for $0 < \beta \leq \frac12$ in \cite[Lemma 9]{CCM}, but in fact it works as well for any $0 < \beta \leq \beta_0$, with the implied constants depending on such $\beta_0$. Here we simply take $\beta_0=1$.

\subsection{Explicit formula} Our next lemma is an (unconditional) explicit formula that connects the non-trivial zeros of $\zeta(s)$ and the prime numbers. 

\begin{lemma}[Guinand-Weil explicit formula] \label{GW}
Let $h(s)$ be analytic in the strip $|\im{s}|\leq \tfrac12+\varepsilon$ for some $\varepsilon>0$, and assume that $|h(s)|\ll(1+|s|)^{-(1+\delta)}$ for some $\delta>0$ when $|\re{s}|\to\infty$. Then
	\begin{align*}
	\displaystyle\sum_{\rho}h\left(\frac{\rho - \tfrac12}{i}\right) & = h\left(\dfrac{1}{2i}\right)+h\left(-\dfrac{1}{2i}\right)-\dfrac{1}{2\pi}\widehat{h}(0)\log\pi+\dfrac{1}{2\pi}\int_{-\infty}^{\infty}h(u)\,\re{\dfrac{\Gamma'}{\Gamma}\!\left(\dfrac{1}{4}+\dfrac{iu}{2}\right)}\,\du \\
	 &  \ \ \ \ \ \ \ \ \ \ \ \ \ -\dfrac{1}{2\pi}\displaystyle\sum_{n\geq2}\dfrac{\Lambda(n)}{\sqrt{n}}\left(\widehat{h}\left(\dfrac{\log n}{2\pi}\right)+\widehat{h}\left(\dfrac{-\log n}{2\pi}\right)\right). 
	\end{align*}
\end{lemma}
\begin{proof} The proof follows from \cite[Theorem 5.12]{IK}. 
\end{proof}

\subsection{Estimates for zeta on the critical strip} Our next lemma provides bounds for the real part of $\zeta'(s)/\zeta(s)$ on the critical strip, under RH. Since these inequalities might be of independent interest, we include both upper and lower bounds for completeness. 

\begin{lemma}\label{th1} Assume RH. For $t\ge 10$, $x\ge 2$, and $0< \beta \le 1$, we have
\begin{align}\label{20231128_18:27}
\begin{split}
&- \left(\frac{1 }{x^{\beta}-1}\right) \log t+ \frac{2 \,x^\beta}{(x^\beta-1)^2} \ \mathrm{Re} \sum_{n\le x} \frac{\Lambda(n)}{n^{1/2+it}} \sinh\!\Big(\beta \log\frac{x}{n} \Big)  
 + O\!\left( \frac{1}{\beta} \left(\frac{\sqrt{x}\log x}{t}+1\right) \right)
 \\
 & \qquad \  \leq -\mathrm{Re} \frac{\zeta'}{\zeta}(\tfrac{1}{2}+\beta+it) \\
 &  \qquad \qquad \ \  \le   \left(\frac{1 }{x^{\beta}+1}\right) \log t+ \frac{2 \,x^\beta}{(x^\beta+1)^2} \ \mathrm{Re} \sum_{n\le x} \frac{\Lambda(n)}{n^{1/2+it}} \sinh\!\Big(\beta \log\frac{x}{n} \Big)  
 + O\!\left( \frac{\beta \sqrt{x}\log x}{t} +1 \right).
\end{split}
\end{align}
\end{lemma}
\begin{proof} To  simplify notation, let $m^{\pm}_{\Delta} = m^{\pm}_{\beta,\Delta}$ and let $x=e^{2\pi \Delta}$.
From \eqref{partial_fraction_dec} and property (ii) of Lemma \ref{lemma_extr_Poisson}, we have
\begin{equation}\label{1st_ineq}
 - \dfrac{1}{2}\log\frac{t}{2\pi} + \sum_{\gamma}m^{-}_{\Delta}(t-\gamma)+O\bigg(\frac{1}{t}\bigg) \leq \mathrm{Re} \, \frac{\zeta'}{\zeta}(\tfrac{1}{2}+\beta+it) \leq - \dfrac{1}{2}\log\frac{t}{2\pi} + \sum_{\gamma}m^{+}_{\Delta}(t-\gamma)+O\bigg(\frac{1}{t}\bigg).
\end{equation}
For a fixed $t >0$, let $\ell^{\pm}_{\Delta}(z):=m^{\pm}_{\Delta}(t-z)$ so that $\widehat{\ell}^{\pm}_{\Delta}(\xi)=\widehat{m}^{\pm}_{\Delta}(-\xi)e^{-2\pi i\xi t}$ and the condition $|\ell^{\pm}_{\Delta}(s)|\ll(1+|s|)^{-2}$ when $|\re{s}|\to\infty$ in the strip $|\im{s}|\leq 1$ follows from \eqref{decay_square_maj_Poisson}, \eqref{Complex_est_maj}, \eqref{Complex_est_min}, and an application of the Phragm\'{e}n-Lindel\"{o}f principle. Recalling that $\widehat{m}^{\pm}_{\Delta}$ are even functions, we apply the Guinand-Weil explicit formula (Lemma \ref{GW}) and find that
\begin{align}\label{GW_applied_to_m}
\begin{split}
\displaystyle\sum_{\gamma}m^{\pm}_{\Delta}(t-\gamma) & = 
\Big\{m^{\pm}_{\Delta}\left(t-\tfrac{1}{2i}\right)+m^{\pm}_{\Delta}\left(t+\tfrac{1}{2i}\right)\Big\}-\dfrac{1}{2\pi}\widehat{m}^{\pm}_{\Delta}(0)\log\pi \\ 
     & \ \ \ \ + \dfrac{1}{2\pi}\int_{-\infty}^{\infty}m^{\pm}_{\Delta}(t-y)\,\re\,\dfrac{\Gamma'}{\Gamma}\bigg(\dfrac{1}{4}+\dfrac{iy}{2}\bigg) \,\dy - \frac{1}{\pi}\sum_{n\geq 2}\dfrac{\Lambda(n)}{\sqrt{n}} \,\widehat{m}^{\pm}_{\Delta}\left(\dfrac{\log n}{2\pi}\right) \cos(t \log n). 
\end{split}
\end{align}
We now proceed with an asymptotic analysis of each term on the right-hand side of \eqref{GW_applied_to_m}.

\subsubsection{First term} From \eqref{Complex_est_maj} and \eqref{Complex_est_min}, for $t\ge 10$ and $x\ge 2$, we see that
\begin{align}\label{AsA_eq1}
\Big|m^{+}_{\Delta}\left(t-\tfrac{1}{2i}\right)+m^{+}_{\Delta}\left(t+\tfrac{1}{2i}\right)\Big| \ll \frac{\Delta^{2}e^{\pi\Delta}}{\beta(1+\Delta t)} \ll \frac{\sqrt{x}\log x}{\beta \, t} 
\end{align}
and
\begin{align}\label{AsA_eq1_min}
\Big|m^{-}_{\Delta}\left(t-\tfrac{1}{2i}\right)+m^{-}_{\Delta}\left(t+\tfrac{1}{2i}\right)\Big| \ll \frac{\beta \Delta^{2}e^{\pi\Delta}}{1+\Delta t} \ll \frac{\beta \sqrt{x}\log x}{ t} .
\end{align}

\subsubsection{Second term} From \eqref{FT_maj_min_Poisson}, it follows that
\begin{equation}\label{conclusion_FT_at_zero_Poisson_maj}
\begin{split}
\widehat{m}^{+}_{\Delta}(0) &= \pi\left( \frac{e^{\pi \beta \Delta} + e^{-\pi \beta \Delta} }{e^{\pi \beta \Delta} - e^{-\pi \beta \Delta} }\right) =  \pi\left( \frac{x^{\beta/2} + x^{-\beta/2}  }{ x^{\beta/2} - x^{-\beta/2}  }\right) =  \pi\left( \frac{x^{\beta} + 1 }{ x^{\beta} -1  }\right)
\end{split}
\end{equation}
and
\begin{equation}\label{conclusion_FT_at_zero_Poisson_min}
\begin{split}
\widehat{m}^{-}_{\Delta}(0) = \pi\left( \frac{e^{\pi \beta \Delta} - e^{-\pi \beta \Delta} }{e^{\pi \beta \Delta} + e^{-\pi \beta \Delta} }\right) =  \pi\left( \frac{x^{\beta/2} - x^{-\beta/2}  }{ x^{\beta/2} + x^{-\beta/2}  }\right) = \pi\left( \frac{x^{\beta} - 1 }{ x^{\beta} +1  }\right).
\end{split}
\end{equation}

\subsubsection{Third term} Recall that the Poisson kernel $h_{\beta}$ defined in \eqref{Def_Poisson_kernel_beta} satisfies $\int_{-\infty}^{\infty} h_{\beta}(y)\,\dy=\pi $. Note also that for $0 < \beta \leq 1$ and $|y| \geq 1$ we have
\begin{equation}\label{h beta}
h_{\beta}(y) = \frac{\beta}{\beta^2 + y^2} \leq \frac{1}{1 + y^2}.
\end{equation}
Hence, from \eqref{decay_square_maj_Poisson} and \eqref{h beta}, we get
\begin{align}\label{aux_pf_Poisson_eq1}
\begin{split}
0 & \leq \int_{-\infty}^{\infty}m^{-}_{\Delta}(y)\,\log(2+|y|) \,\dy \\
& \leq \int_{-\infty}^{\infty} h_{\beta}(y) \log (2 + |y|)\,\dy  = \int_{-1}^{1}h_{\beta}(y) \log (2 + |y|) \,\dy + \int_{|y| \ge 1}h_{\beta}(y) \log (2 + |y|) \,\dy = O(1).
\end{split} 
\end{align}
From Stirling's formula, \eqref{conclusion_FT_at_zero_Poisson_min}, and \eqref{aux_pf_Poisson_eq1}, for $t \geq 10$, it follows that 
\begin{align}\label{Conclusion_poisson_log_sum_min}
\begin{split}
\dfrac{1}{2\pi}\int_{-\infty}^{\infty}m^{-}_{\Delta}(t-y)\,\re\,\dfrac{\Gamma'}{\Gamma}\bigg(\dfrac{1}{4}+\dfrac{iy}{2}\bigg) \,\dy  &= \dfrac{1}{2\pi}\int_{-\infty}^{\infty}m^{-}_{\Delta}(y)\Big(\log \tfrac{t}{2}+ O(\log(2+|y|))\Big) \,\dy \\
& = \dfrac{1}{2} \left( \frac{x^{\beta} - 1 }{x^{\beta} + 1  }\right) \log \frac{t}{2} + O(1).  
\end{split}
\end{align}
Similarly, using Stirling's formula, \eqref{decay_square_maj_Poisson}, and \eqref{conclusion_FT_at_zero_Poisson_maj}, for $t \geq 10$, it follows that 
\begin{align}\label{Conclusion_poisson_log_sum_maj}
\begin{split}
\dfrac{1}{2\pi}\int_{-\infty}^{\infty}m^{+}_{\Delta}(t-y)\,\re\,\dfrac{\Gamma'}{\Gamma}\bigg(\dfrac{1}{4}+\dfrac{iy}{2}\bigg) \,\dy  &= \dfrac{1}{2\pi}\int_{-\infty}^{\infty}m^{+}_{\Delta}(y)\Big(\log \tfrac{t}{2}+ O(\log(2+|y|))\Big) \,\dy \\
& = \dfrac{1}{2} \left( \frac{x^{\beta} + 1 }{x^{\beta} - 1  }\right) \log \frac{t}{2} + O\Big(\frac{1}{\beta}\Big).  
\end{split}
\end{align}

\subsubsection{Fourth term} Now, to handle the sum over prime powers, recall that  $x = e^{2\pi \Delta}$ and note that the sum in \eqref{GW_applied_to_m} only runs over $n$ with $2 \leq n \leq x$. Using the explicit description for the Fourier transforms $\widehat{m}_{\Delta}^{\pm}$ given by \eqref{FT_maj_min_Poisson}, we get
\begin{equation}\label{Poisson_lem_expres_antes}
\begin{split}
\frac{1}{\pi}\sum_{n\geq 2}\dfrac{\Lambda(n)}{\sqrt{n}} \,\widehat{m}^{\pm}_{\Delta}\left(\dfrac{\log n}{2\pi}\right) \cos(t \log n) &= \frac{e^{2\pi \beta \Delta}}{\big(e^{2\pi \beta \Delta} \mp 1\big)^2} \ \mathrm{Re}\!\!\!\! \sum_{n\leq e^{2\pi\Delta}}\dfrac{\Lambda(n)}{n^{1/2+it}}\,\bigg(\dfrac{e^{2\pi\beta\Delta}}{n^\beta}-\dfrac{n^\beta}{e^{2\pi\beta\Delta}}\bigg)
\\
&=  \frac{ x^{\beta}}{\big( x^{\beta} \mp 1 \big)^2}\ \mathrm{Re} \sum_{n\leq x}\dfrac{\Lambda(n)}{n^{1/2+it}}\,\left\{ \Big(\dfrac{x}{n}\Big)^\beta-\Big(\dfrac{x}{n}\Big)^{-\beta} \right\}
\\
&= \frac{ 2\,x^{\beta}}{\big( x^{\beta} \mp 1 \big)^2} \ \mathrm{Re} \sum_{n\leq x}\dfrac{\Lambda(n)}{n^{1/2+it}}\, \sinh\!\Big( \beta \log \frac{x}{n} \Big).
\end{split}
\end{equation}

\subsubsection{Conclusion} To derive the bound on the right-hand side of \eqref{20231128_18:27}, we combine \eqref{1st_ineq}, \eqref{GW_applied_to_m}, \eqref{AsA_eq1_min}, \eqref{conclusion_FT_at_zero_Poisson_min}, \eqref{Conclusion_poisson_log_sum_min},  and \eqref{Poisson_lem_expres_antes}, to find that
\[
\begin{split}
\mathrm{Re} \frac{\zeta'}{\zeta}(\tfrac{1}{2}+\beta+it) &\ge  - \left(\frac{1 }{x^{\beta}+1}\right) \log\Big(\frac{t}{2\pi} \Big) \\
& \qquad \qquad-\frac{ 2\,x^{\beta}}{\big( x^{\beta} + 1 \big)^2} \ \mathrm{Re} \sum_{n\leq x}\dfrac{\Lambda(n)}{n^{1/2+it}}\, \sinh\!\Big( \beta \log \frac{x}{n} \Big)
 + O\!\left( \frac{\beta \sqrt{x}\log x}{ t} +1\right).
\end{split}
\]
To derive the bound on the left-hand side of \eqref{20231128_18:27}, we combine \eqref{1st_ineq}, \eqref{GW_applied_to_m}, \eqref{AsA_eq1}, \eqref{conclusion_FT_at_zero_Poisson_maj}, \eqref{Conclusion_poisson_log_sum_maj},  and \eqref{Poisson_lem_expres_antes}, to conclude that
\[
\begin{split}
\mathrm{Re} \frac{\zeta'}{\zeta}(\tfrac{1}{2}+\beta+it) &\le   \left(\frac{1 }{x^{\beta}-1}\right) \log\Big(\frac{t}{2\pi} \Big)  \\
&\qquad  \qquad - \frac{ 2\,x^{\beta}}{\big( x^{\beta} - 1 \big)^2} \ \mathrm{Re} \sum_{n\leq x}\dfrac{\Lambda(n)}{n^{1/2+it}}\, \sinh\!\Big( \beta \log \frac{x}{n} \Big) 
 + O\!\left( \frac{1}{\beta} \left(\frac{\sqrt{x}\log x}{t}+1\right) \right).
\end{split}
\]
This completes the proof of Lemma \ref{th1}.
\end{proof}

\section{Proof of Theorem \ref{new_Sound}}

For $t \ge 10$, observe that 
\begin{equation}\label{20231129_12:35}
\begin{split}
\log|\zeta(\tfrac{1}{2}+it)| &= -\int_{0}^{1} \re \frac{\zeta'}{\zeta}(\tfrac{1}{2}+ u + it) \, \mathrm{d}u + \log|\zeta(\tfrac32+it)| = -\int_{0}^{1} \re \frac{\zeta'}{\zeta}(\tfrac{1}{2}+ u + it) \, \mathrm{d}u + O( 1).
\end{split}
\end{equation}
Since
\[
\frac{\mathrm{d}}{\mathrm{d}u} \left[  \frac{\log\big(1+ x^{-u}\big)}{\log x} \right] = - \frac{ 1}{x^{u} + 1},
\]
we note that
\begin{align}\label{20231129_12:26}
 \log t \int_0^{1} \frac{ \mathrm{d}u }{x^{u} + 1} \le  \log t \int_0^{\infty} \frac{ \mathrm{d}u }{x^{u} + 1}  =  \log 2 \cdot \frac{\log t}{\log x}.
\end{align}

\smallskip

Applying the upper bound from Lemma \ref{th1} in \eqref{20231129_12:35}, and using \eqref{20231129_12:26}, we get 
\begin{equation}\label{20231228_17:31}
\begin{split}
\log|\zeta(\tfrac{1}{2}+it)| &\le   \log t \int_0^{1} \frac{ \mathrm{d}u} {x^{u} + 1}  +  \mathrm{Re} \sum_{n\leq x}\dfrac{\Lambda(n)}{n^{1/2+it}} W_0(n;x) + O\!\left( \frac{\sqrt{x}\log x}{t} +1\right)\\
& \leq \log 2 \cdot \frac{\log t}{\log x} +  \mathrm{Re} \sum_{n\leq x}\dfrac{\Lambda(n)}{n^{1/2+it}} W_0(n;x) + O\!\left( \frac{\sqrt{x}\log x}{t} +1\right),
\end{split}
\end{equation}
where
\begin{equation*}
W_0(n;x) = \int_0^{1} \frac{ 2\,x^{u}}{\big( x^{u} + 1 \big)^2} \sinh\!\Big( u \log \frac{x}{n} \Big) \, \mathrm{d}u.
\end{equation*}
Now observe that 
\begin{align*}
\begin{split}
W_0(n,x) &= \frac{1}{2} \int_0^{\infty}  \frac{  \sinh\!\Big( u \log \frac{x}{n} \Big)}{\cosh^2(\frac{u}{2}\log x)} \, \mathrm{d}u + O\!\left( \int_{1}^\infty \frac{\mathrm{d}u}{n^u} \right)  = \frac{1}{\log x} \int_0^{\infty}  \frac{  \sinh\!\Big( 2y  \frac{\log (x/n)}{\log x} \Big)}{\cosh^2(y)} \, \mathrm{d}y + O\!\left( \frac{1}{n\log n} \right)\\
&=  \frac{1}{\log x} \cdot F\!\left( \frac{\log(x/n)}{\log x} \right) +  O\!\left( \frac{1}{n\log n} \right),
\end{split}
\end{align*}
with $F$ as in \eqref{20231127_13:35}. Therefore
\begin{align}\label{20231206_14:16}
\mathrm{Re} \sum_{n\leq x}\dfrac{\Lambda(n)}{n^{1/2+it}} W_0(n;x)  = \mathrm{Re} \sum_{n\leq x}\dfrac{\Lambda(n)}{n^{1/2+it} \log x} F\!\left( \frac{\log(x/n)}{\log x} \right)  + O(1). 
\end{align}
Theorem \ref{new_Sound} now follows directly from \eqref{20231228_17:31} and \eqref{20231206_14:16}.

\section{Proof of Theorem \ref{lower-order}}

We now deduce Theorem \ref{lower-order} from Theorem \ref{new_Sound}.

\subsection{Setup}
Assuming RH, we have
\begin{equation}\label{prime_est}
\sum_{n\le x} \frac{\Lambda(n)}{\sqrt{n}} = 2 \sqrt{x} + O\big( \log^3 x\big). 
\end{equation}
Since $F(0)=0$ and $F(u) \ge 0$ for $u \in [0,1)$, using \eqref{prime_est} and summing by parts, it follows that
\[
\begin{split}
\bigg| \ \mathrm{Re} \sum_{n\leq x}\dfrac{\Lambda(n)}{n^{1/2+it} \log x} \, F\!\left( \frac{\log(x/n)}{\log x} \right) \, \bigg| \ &\le \ \frac{1}{\log x} \sum_{n\le x} \frac{\Lambda(n)}{\sqrt{n}} F\!\left( \frac{\log(x/n)}{\log x} \right)
\\
&= \frac{2}{\log x} \int_{\log 2/\log x}^1  F'( 1 - v) \, x^{v/2} \, \dv  + O\big( \log^2 x\big)
\\
&= \frac{2}{\log x} \int_{1/2}^1  F'( 1 - v) \, x^{v/2} \, \dv  + O\big( x^{1/4} \log x \big)
\\
&=  \frac{2 \sqrt{x} }{\log x} \int_{0}^{1/2}  F'(u) \, x^{-u/2} \, \du  + O\big( x^{1/4} \log x \big).
\end{split}
\]
Here we have used the fact that $F'(u)$ has a double pole at $u=1$ so that $|F(1-\frac{\log y}{\log x})| \ll \frac{\log^2 x}{\log^2 y}$ for $2\le y \le x$. Differentiating the series expansion for $F$ in \eqref{series} term-by-term, we see that
\[
F'(u) = 2 \log 2 + 2 \sum_{k=1}^\infty \left(1\!-\!\frac{1}{2^{2k}}\right) (2k\!+\!1) \, \zeta(2k\!+\!1) \, u^{2k}
\]
uniformly for $|u|\le \frac{1}{2}$. Notice that the coefficients of this series are positive. Hence, for any positive integer $K$, we have
\begin{equation}\label{est1}
\begin{split}
\bigg| \ \mathrm{Re} \sum_{n\leq x} &\dfrac{\Lambda(n)}{n^{1/2+it} \log x} \, F\!\left( \frac{\log(x/n)}{\log x} \right) \, \bigg|
\\ &\le \frac{(4 \log 2) \sqrt{x}}{\log x} \int_{0}^{1/2}  x^{-u/2} \, \du + \frac{4 \sqrt{x} }{\log x}  \sum_{k=1}^K \left(1\!-\!\frac{1}{2^{2k}}\right)  (2k\!+\!1) \, \zeta(2k\!+\!1) \, \int_{0}^{1/2}  u^{2k} \, x^{-u/2} \, \du  \\
& \quad \qquad + O_K\bigg( \frac{\sqrt{x} }{\log x}  \int_{0}^{1/2}  u^{2K+2} \, x^{-u/2} \, \du \bigg) + O\big( x^{1/4} \log x \big).
\end{split}
\end{equation}
For non-negative integers $k$, we observe that 
\begin{equation}\label{est2}
 \int_{0}^{1/2}  u^{2k} \, x^{-u/2} \, \du \le  \int_{0}^{\infty}  u^{2k} \, x^{-u/2} \, \du =  \frac{ 2^{2k+1} \, \Gamma(2k\!+\!1) }{(\log x)^{2k+1}} =  \frac{2^{2k+1} (2k)! }{(\log x)^{2k+1}}.
\end{equation}
Therefore, combining \eqref{main_ineq}, \eqref{est1}, and \eqref{est2}, it follows that
\begin{equation}\label{ineq2}
\begin{split}
\log\big|\zeta(\tfrac{1}{2}+it)\big| & \le \log 2 \cdot \frac{\log t}{\log x} + \frac{(8 \log 2) \sqrt{x}}{\log^2 x} + 8 \sqrt{x}  \sum_{k=1}^K \Big(2^{2k} \!-\! 1\Big) \,  \frac{ (2k\!+\!1)! \, \zeta(2k\!+\!1) }{(\log x)^{2k+2}}  
\\
& \qquad + O_K\bigg( \frac{\sqrt{x} }{(\log x)^{2K+4}}  \bigg).
\end{split}
\end{equation}

\subsection{Choosing $x$: initial approximations}
We roughly want to choose $\sqrt{x} \approx \log t$ in order to minimize the right-hand side of \eqref{ineq2}. If we make the exact choice $\sqrt{x} = \log t$, we obtain
\begin{equation}\label{20240102_13:27}
\begin{split}
\log\big|\zeta(\tfrac{1}{2}+it)\big| & \le \frac{\log 2}{2} \frac{\log t}{\log \log t} +  \frac{ 2 \log 2 \cdot \log t}{(\log \log t)^2} + 2 \log t \sum_{k=1}^K \Big(1\!-\!\frac{1}{2^{2k}}\Big)\,  \frac{ (2k\!+\!1)! \, \zeta(2k\!+\!1) }{(\log \log t)^{2k+2}}  
\\
& \qquad + O_K\bigg( \frac{\log t }{(\log \log t)^{2K+4}}  \bigg).
\end{split}
\end{equation}
We can make the second term on the right-hand side of this inequality smaller by taking $x$ to be slightly smaller. To this end, we choose
\begin{equation} \label{x size}
\log x = 2 \log \log t - 2c
\end{equation}
in \eqref{ineq2}, and find that
\begin{equation} \label{step1}
\begin{split}
\log\big|\zeta(\tfrac{1}{2}+it)\big| & \, \le \, \frac{\log 2}{2} \frac{\log t}{\log \log t} + \bigg( \frac{c \log 2}{2} + 2 \, e^{-c} \log 2 \bigg)  \frac{\log t}{(\log \log t)^2}  
\\
& \qquad \quad + \bigg( \frac{c^2 \log 2}{4} + 4 \, c \, e^{-c} \log 2 \bigg)  \frac{\log t}{(\log \log t)^3} + O_c\bigg( \frac{\log t }{(\log \log t)^{4}}  \bigg).
\end{split}
\end{equation}
The minimum value of the function $g(c) =\frac{c \log 2}{2} + 2 \, e^{-c} \log 2$
occurs at $c=2 \log 2$ with the minimum value equaling $\frac{\log 2}{2} + \log^22$. Choosing $c=2 \log 2$ in \eqref{step1}, we conclude that
\[
\begin{split}
\log\big|\zeta(\tfrac{1}{2}+it)\big| & \, \le \, \frac{\log 2}{2} \frac{\log t}{\log \log t} + \bigg(\frac{\log 2}{2} + \log^2 2\bigg) \frac{\log t }{(\log \log t)^2} 
 + \Big( 2 \log^2 2 + 2 \log^3 2 \Big) \frac{\log t }{(\log \log t)^3} 
 \\
& \quad \qquad + O\bigg( \frac{\log t }{(\log \log t)^{4}}  \bigg),
\end{split}
\]
arriving at the first three terms stated in Theorem \ref{lower-order}.
\subsection{The optimal choice of $x$} \label{Optimal_x}We now describe the optimal choice of $x$ in terms of $t$, as hinted by \eqref{x size}, but now with a complete power series expansion. The elegant argument presented in this subsection, and the companion Pari script, are due to Don Zagier \footnote{ In an earlier version of this paper we had a slightly different line of reasoning for this description, leading to the same coefficients.}.

\smallskip

We make the change of variables $w=\frac1{\log\log t}$ and  $z=\frac1{\log x}$.  Define
$$ B(w,z) \= Lz\t e^{1/w} \+ e^{1/2z}\t\sum_{m=0}^\infty a_m z^{m+1} $$
as in \eqref{ineq2}, where the coefficients $L$ and $a_m$ could be arbitrary numbers but in our case are given by 
$$ L=\log 2,\quad a_1=8L, \quad a_m=8\t(2^{m-1}-1)\t m!\t \zeta(m)\;\,\text{for $m>1$~odd},\quad a_m=0\;\,\text{for $m$ even}. $$
We want an extreme value of~$B$ for~$w$ fixed, so need  $w$ and $z$ related by
$$ 0 \= \frac{\partial B}{\partial z} \= L\t e^{1/w} \+ e^{1/2z}\t\sum_{m=0}^\infty \bigl((m+1)\t a_m \m\tf12\t a_{m+1}\bigr)\t z^m\,. $$
This can be rewritten in turn as
\begin{equation}\label{z_Don2}
e^{1/w} \= e^{1/2z}\t\sum_{m=0}^\infty b_m\t z^m,\qquad b_m\,:=\,\frac{\frac12\t a_{m+1}\m (m+1)\t a_m}L 
\end{equation}
or, taking logarithms, 
$$ \frac1w\= \frac1{2z} \+ \log b_0 \+ \log\left(1\+ \sum_{m=1}^\infty \frac{b_m}{b_0}\t z^m\right) $$
(with $b_0=4$ and $\log b_0=2L$ in our case) or, inverting in the sense of multiplication of power series,
$$  w \= 2z \m 4\t\log b_0\t z^2 \+ \left(8\t\log^2 b_0\m\frac{4b_1}{b_0}\right)\t z^3 \+\cdots $$
or, inverting in the sense of composition of power series,
\begin{equation}\label{z_Don}
z \= \frac12\t w \+ \frac{\log b_0}2\t w^2 \+ \left(\frac12\log^2 b_0 + \frac{b_1}{4b_0}\right)\t w^3 \+ \cdots  
\end{equation}
or, inverting once more in the sense of multiplication of power series,
$$  \frac1z \= \frac2w \m 2\tt\log b_0  \m \frac{b_1}{b_0}\t w \+ \cdots $$
(this generalizes \eqref{x size}), each time with explicitly computable coefficients that are polynomials in the numbers $a_m$, $1/L$, $\log b_0$, 
and $1/b_0$ (so, polynomials with rational coefficients in $(\log 2)^{\pm1}$ and $\zeta(2k+1)$ in our case).  Finally, substituting
the value of~$z$ as in \eqref{z_Don2} and \eqref{z_Don} into the function $B(w,z)$, we find
\begin{align*}
B_{\text{opt}}(w) &\= e^{1/w}\t\left(Lz \+ \sum_{m=0}^\infty a_m z^{m+1}\t\Big/\sum_{m=0}^\infty b_m\t z^m \right)  \\
  &\= e^{1/w}\t\left( \left(L+\frac{a_0}{b_0}\right)\t z \+ \left(\frac{a_1}{b_0}-\frac{a_0b_1}{b_0^2}\right)\t z^2 \+ \cdots\right)   \\
    &\= e^{1/w}\t\left(\frac12\t\left(L+\frac{a_0}{b_0}\right)\t w\+ \left(\frac14\t\left(\frac{a_1}{b_0}-\frac{a_0b_1}{b_0^2}\right)
       +\frac{\log b_0}2\t\left(L+\frac{a_0}{b_0}\right)\right)\t w^2 \+ C_3\tt w^3 \+ \cdots\right),  
\end{align*}       
which is the desired form proposed in Theorem \ref{lower-order}.

\smallskip

The following Pari script calculates the first coefficients of this series.

\begin{verbatim}
a(m) = [8*L,0,144*Z3,0,14400*Z5,0,2540160*Z7,0][m];\\def. of a_m for m=1,...,8; a_0 = 0
b(m) = (a(m+1)/2 - (m+1)*a(m))/L; 
w1 = 1/2/z + 2*L + log(1+ sum(m=1,6,b(m)/4*z^m,O(z^7))); Z = reverse(1/w1); \\ w1 = 1/w
B = L*Z + sum(m=1,8,a(m)*Z^(m+1),O(z^10))/sum(m=1,7,b(m)*Z^m,4+O(z^8)); 

\\ Time 7 ms.  Output:

w1 == 1/2/z + 2*L - 4*z + (-8 + 18*Z3/L)*z^2 + (-64/3 - 72*Z3/L)*z^3 + O(z^4) \\ true 
Z == 1/2*z + L*z^2 + (2*L^2 - 1)*z^3 + (4*L^3 - 6*L - 1 + 9/4*Z3/L)*z^4 + O(z^5) \\ true
coeff(B,1) == L/2
coeff(B,2) == L^2 + L/2
coeff(B,3) == 2*L^3 + 2*L^2 
coeff(B,4) == 4*L^4 + 6*L^3 - L + 9/4*Z3 
coeff(B,5) == 8*L^5 + 16*L^4 - 8*L^2 - 4/3*L + (18*L + 9/2)*Z3
coeff(B,6) == 16*L^6 + 40*L^5 - 40*L^3 - 40/3*L^2  + 4/3*L+ (90*L^2+45*L-9)*Z3 
      - 81/16*Z3^2/L + 225/4*Z5
coeff(B,7) == 32*L^7 + 96*L^6 - 160*L^4 - 80*L^3 + 16*L^2 + 34/5*L 
     + 45*(8*L^3 + 6*L^2 - 12/5*L - 1)*Z3 - 81/4*(3-1/L)*Z3^2 + (675*L + 225/2)*Z5
\end{verbatim}

\section*{Acknowledgements} We thank Don Zagier for the argument presented in \S \ref{Optimal_x}. EC thanks the organizers of the conference {\it Constructive Theory of Functions 2023}, in Lozenets, Bulgaria, for their warm hospitality. MBM was supported by the NSF grant DMS-2101912.

\end{document}